\documentclass{article}
\usepackage[cp1251]{inputenc}
\usepackage[russian]{babel}
\usepackage{amssymb,amsfonts,amsmath,amsthm,amscd}
\usepackage{graphicx}
\usepackage{enumerate}
\usepackage{soul}
\usepackage[matrix,arrow,curve]{xy}
\usepackage{longtable}
\usepackage{array}
\RequirePackage{russlh}
\RequirePackage{mathlh}

\newdimen\symskip
\newdimen\defskip
\newdimen\parind
\parind=\parindent
\newdimen\leftmarge
\newdimen\theoremshape
\topsep\defskip
\makeatletter
\newcommand*{\клей}{\nobreak\hskip\z@skip}
\newcommand{\?}{\,\nobreak\hskip0pt}

\renewcommand{\:}{\textup{:}}
\renewcommand{\~}{\textup{;}}
\DeclareRobustCommand*{\т}{~\textemdash{} }
\DeclareRobustCommand*{\д}{\клей\hbox{-}\клей}
\newcommand{\no}{}
\renewcommand{\@listI}{\settowidth\labelwidth{\labheadi{\no}}\listipar{\parind}{\labelwidth}}
\newcommand{\listivpar}{\topsep\defskip\partopsep0pt\parsep-\parskip\itemsep0.5\topsep}
\newcommand{\listipar}[2]{\rightmargin0pt\leftmargin#1\labelsep#1\advance\labelsep-#2\itemindent0pt\listivpar}
\renewcommand{\@listii}{\settowidth\labelwidth{\labheadii{\@roman{\no}}}\listiipar{\parind}{\labelwidth}}
\newcommand{\listiivpar}{\topsep0.5\defskip\partopsep0pt\parsep-\parskip\itemsep0.5\topsep}
\newcommand{\listiipar}[2]{\rightmargin0pt\leftmargin#1\labelsep#1\advance\labelsep-#2\itemindent0pt\listiivpar}
\def\thempfn{\ifcase\value{footnote}1\or *\or **\or ***\else\@ctrerr\fi}
\renewcommand\footnoterule{%
  \kern-3\p@
  \hrule\@width1in
  \kern2.6\p@}
\makeatother

\makeatletter

\renewcommand{\@biblabel}[1]{[#1]}
\renewenvironment{thebibliography}[1]
     {\renewcommand{\refname}{Литература}%
      \section*{\refname}%
      \@mkboth{\MakeUppercase\refname}{\MakeUppercase\refname}%
      \list{\@biblabel{\@arabic\c@enumiv}}%
           {\itemsep\baselineskip
            \leftmargin\parind
            \settowidth\labelwidth{\@biblabel{#1}}%
            \labelsep\parind\advance\labelsep-\labelwidth
            \@openbib@code
            \usecounter{enumiv}%
            \let\p@enumiv\@empty
            \renewcommand\theenumiv{\@arabic\c@enumiv}}%
      \sloppy
      \clubpenalty4000
      \@clubpenalty\clubpenalty
      \widowpenalty4000%
      \sfcode`\.\@m}
     {\def\@noitemerr
       {\@latex@warning{Empty `thebibliography' environment}}%
      \endlist}

\def\@maketitle{%
  \newpage
  \vskip0.5em%
  УДК \udk%
  \vskip0.5em%
  MSC \msc%
  \vskip1em%
  \begin{center}\bf%
  \let\footnote\thanks%
   {\Large\@author\par}%
   \vskip1.5em%
   {\LARGE\@title\par}%
   \vskip1em%
   {\large\@date}%
  \end{center}%
  \par
  \vskip1.5em}

\makeatother

\makeatletter

\renewcommand\sectionmark[1]{%
 \markright{%
  \ifnum \c@secnumdepth >\z@
   \thesection. \ %
  \fi
 #1}}%

\renewcommand{\section}{\@startsection{section}{1}{0pt}%
{5.5ex plus .5ex minus .2ex}{1.5ex plus .3ex}%
{\center\normalfont\Large\bfseries\sffamily\bom}}
\renewcommand{\subsection}{\@startsection{subsection}{2}{0pt}%
{4.5ex plus .4ex minus .2ex}{0.75ex plus .2ex}%
{\center\normalfont\large\bfseries\sffamily\bom}}
\renewcommand{\subsubsection}{\@startsection{subsubsection}{3}{0pt}%
{2.5ex plus .5ex minus .2ex}{1ex plus .2ex}%
{\center\normalfont\bfseries\scshape\bom}}

\newcommand{\Ss}{\textup{\S\,}}
\newcommand{\Sss}{\textup{\S\S\,}}
\def\@postskip@{\hskip.5em\relax}

\def\postsection{.\@postskip@}

\def\postsubsection{.\@postskip@}

\def\postsubsubsection{.\@postskip@}

\def\postparagraph{.\@postskip@}

\def\postsubparagraph{.\@postskip@}

\def\@seccntformat#1{\csname pre#1\endcsname\csname the#1\endcsname\csname post#1\endcsname}

\makeatother

\renewcommand{\thesection}{\textup{\arabic{section}}}

\newcommand{\parr}{\par\addvspace{\defskip}}
\newcommand{\theo}[2]{\newtheorem{#1}{#2}[section]}
\newcommand{\deff}[2]{\newenvironment{#1}{\parr\textbf{#2.}}{\parr}}
\theo{cas}{Случай}
\theo{problem}{Проблема}
\theo{theorem}{Теорема}
\theo{lemma}{Лемма}
\theo{prop}{Предложение}
\theo{stm}{Утверждение}
\theo{imp}{Следствие}
\theo{ex}{Пример}
\deff{df}{Определение}
\deff{note}{Замечание}
\deff{denote}{Обозначение}
\deff{denotes}{Обозначения}
\deff{hint}{Указание}
\deff{answer}{Ответ\:}
\newenvironment{cass}[1]{\begin{cas}\label{#1}\upshape}{\end{cas}}

\makeatletter
\def\@begintheorem#1#2[#3]{%
  \deferred@thm@head{\the\thm@headfont \thm@indent
    \@ifempty{#1}{\let\thmname\@gobble}{\let\thmname\@iden}%
    \@ifempty{#2}{\let\thmnumber\@gobble}{\let\thmnumber\@iden}%
    \@ifempty{#3}{\let\thmnote\@gobble}{\let\thmnote\@iden}%
    \thm@notefont{\bfseries\upshape}%
    \indent%
    \thm@swap\swappedhead\thmhead{#1}{#2}{#3}%
    \the\thm@headpunct
    \thmheadnl 
    \hskip\thm@headsep
  }%
  \ignorespaces}
\renewenvironment{proof}{\parr\pushQED{\qed}\normalfont$\square\quad$}{\popQED\@endpefalse\parr}

\makeatother

\newcommand{\labheadu}[1]{\textup{#1.}}
\newcommand{\labheadi}[1]{\textup{#1)}}
\newcommand{\labheadii}[1]{\textup{(#1)}}
\newcommand{\labhi}[1]{\labheadi{\arabic{#1}}}
\newcommand{\labhii}[1]{\labheadii{\roman{#1}}}

\newenvironment{nums}[1]{\renewcommand{\no}{#1}\begin{enumerate}}{\end{enumerate}}
\newcommand{\eqn}[1]{\begin{equation}#1\end{equation}}

\makeatletter

\def\LT@makecaption#1#2#3{%
  \LT@mcol\LT@cols c{\hbox to\z@{\hss\parbox[t]\LTcapwidth{%
    \sbox\@tempboxa{#1{#2. }#3}%
    \ifdim\wd\@tempboxa>\hsize
      #1{#2. }#3%
    \else
      \hbox to\hsize{\hfil\box\@tempboxa\hfil}%
    \fi
    \endgraf\vskip\baselineskip}%
  \hss}}}

\@addtoreset{equation}{section}
\@addtoreset{footnote}{section}

\newcommand{\high}{\vph{\Br{}_0^0}}
\newcounter{numt}
\newcounter{col}
\newcounter{coll}
\newcommand{\mt}[3]{\multicolumn{#1}{#2}{#3}}
\renewcommand{\thenumt}{\textup{\arabic{numt})}}
\newcommand{\news}{\\\hline}
\newcommand{\refs}{\refstepcounter{numt}}

\newcommand{\an}[1]{\refs\label{#1}\thenumt&}
\newcommand{\n}[1]{\news\an{#1}}
\newcommand{\refcol}[1]{\addtocounter{col}{#1}}
\newcommand{\fmc}{\news\mt{1}{|c|}{$\high$ №}}
\newcommand{\mc}[2]{&\mt{#1}{>{$}c<{$}@{\refcol{#1}}|}{#2}}
\newcommand{\nc}[1]{\mc{1}{#1}}
\newcommand{\mtc}[1]{\mt{\value{coll}}{c}{#1}}
\newcommand{\lont}[5]{%
\setcounter{col}{1}\setcounter{numt}{0}
\begin{longtable}{|r|#1}\caption{#2}\label{#3}\fmc#4
\setcounter{coll}{\value{col}}\news\endfirsthead
\mtc{Продолжение табл.~\ref{#3}.}\\
\mtc{}
\fmc#4\news\endhead
\mtc{}\\
\mtc{\textit{Продолж. на сл. стр.}}
\endfoot
\endlastfoot
#5\news\end{longtable}}

\makeatother

\renewcommand{\ge}{\geqslant}
\renewcommand{\le}{\leqslant}
\newcommand{\fa}{\,\forall\,}
\newcommand{\exi}{\,\exists\,}

\newcommand{\es}{\varnothing}
\newcommand{\subs}{\subset}

\newcommand{\sm}{\setminus}
\newcommand{\cln}{\colon}

\newcommand{\Ra}{\Rightarrow}

\newcommand{\os}[1]{\overset{#1}}
\newcommand{\Inn}[1]{\smash{\os{\circ}{\smash{#1}\vph{^{_{^{_{c}}}}}}}\vph{#1}}
\newcommand{\wt}{\widetilde}

\newcommand*{\bw}[1]{#1\nobreak\discretionary{}{\hbox{$\mathsurround=0pt #1$}}{}}
\newcommand{\sco}{,\ldots,}

\newcommand{\seq}{\bw=\ldots\bw=}

\newcommand{\sop}{\bw\oplus\ldots\bw\oplus}

\newcommand{\sot}{\bw\otimes\ldots\bw\otimes}

\newcommand{\ha}[1]{\left\langle#1\right\rangle}
\newcommand{\ba}[1]{\bigl\langle#1\bigr\rangle}

\newcommand{\br}[1]{\bigl(#1\bigr)}
\newcommand{\Br}[1]{\Bigl(#1\Bigr)}

\newcommand{\ter}[1]{\textup{(}#1\textup{)}}

\newcommand{\bc}[1]{\bigl\{#1\bigr\}}

\newcommand{\mbb}{\mathbb}
\newcommand{\mbf}{\mathbf}
\newcommand{\mcl}{\mathcal}
\newcommand{\mfr}{\mathfrak}

\newcommand{\R}{\mbb{R}}

\newcommand{\N}{\mbb{N}}

\newcommand{\Cbb}{\mbb{C}}
\newcommand{\Hbb}{\mbb{H}}

\newcommand{\Eb}{\mbb{E}}

\newcommand{\ib}{\mbf{i}}

\newcommand{\Pc}{\mcl{P}}

\newcommand{\ggt}{\mfr{g}}
\newcommand{\hgt}{\mfr{h}}
\newcommand{\tgt}{\mfr{t}}
\newcommand{\pd}{\partial}

\newcommand{\al}{\alpha}
\newcommand{\be}{\beta}

\newcommand{\de}{\delta}
\newcommand{\De}{\Delta}

\newcommand{\la}{\lambda}
\newcommand{\La}{\Lambda}

\newcommand{\ph}{\varphi}

\newcommand{\Om}{\Omega}

\DeclareMathOperator{\ad}{ad}
\DeclareMathOperator{\Mat}{Mat}

\DeclareMathOperator{\rk}{rk}

\newcommand{\GL}{\mbf{GL}}

\newcommand{\glg}{\mfr{gl}}
\newcommand{\slg}{\mfr{sl}}
\newcommand{\sog}{\mfr{so}}
\newcommand{\un}{\mfr{u}}
\newcommand{\sug}{\mfr{su}}
\newcommand{\spg}{\mfr{sp}}

\newcommand{\bom}{\boldmath}
\newcommand{\vph}[1]{\vphantom{#1}}

\newcommand{\thra}{\twoheadrightarrow}

\begin{document}

\author{О.\,Г.\?Стырт}
\title{О~простейших стационарных подалгебрах\\
для компактных линейных алгебр Ли}
\date{}
\newcommand{\udk}{512.815.1+512.815.6+512.816.1}
\newcommand{\msc}{22E46+17B10+17B20+17B45}

\maketitle

{\leftskip\parind\rightskip\parind
Получены некоторые достаточные условия существования вектора с~одномерной либо простой трёхмерной стационарной подалгеброй для неприводимой компактной линейной алгебры Ли.

\smallskip

\textbf{Ключевые слова\:} компактная линейная алгебра Ли, система корней, схема Дынкина, стационарная подалгебра общего положения (с.п.о.п.).

\smallskip

Some sufficient conditions of existence of a~vector with one-dimensional or simple three-dimensional stationary subalgebra for an irreducible compact linear Lie algebra are obtained.

\smallskip

\textbf{Key words\:} compact linear Lie algebra, root system, Dynkin scheme, stationary subalgebra in general position (s.s.g.p.).\par}

\section{Введение}\label{introd}

Пусть $V_{\R}$\т вещественное векторное пространство, а~$\ggt_{\R}\subs\slg(V_{\R})$\т неприводимая компактная линейная алгебра Ли ранга $r>1$, имеющая простой коммутант.

Нас будет интересовать вопрос о~том, существует ли вектор $v\in V_{\R}$, стационарная подалгебра которого изоморфна одной из алгебр $\R$ и~$\sug_2$, т.\,е. выполняется ли условие
\eqn{\label{elem}
\exi v\in V_{\R}\cln\quad\quad\rk\br{(\ggt_{\R})_v}=1.}

Данная задача является вспомогательной для проблемы нахождения всех компактных линейных групп, топологический фактор действия которых гомеоморфен клетке. К~настоящему моменту разобраны случаи группы с~коммутативной связной компонентой~\cite{My1} и~простой трёхмерной группы~\cite{My2,My3}. Произвольное же линейное представление, обладающее вектором со стабилизатором ранга~$1$, планируется свести к~соответствующему слайс-представлению при помощи теоремы о~слайсе~\cite[гл.\,II,~\Sss4---5]{Bredon}.

Мы будем пользоваться следующими обозначениями, связанными с~представлениями комплексных редуктивных алгебр Ли. Представление произвольной комплексной редуктивной алгебры Ли, сопряжённое представлению~$R$, обозначим через~$R'$. При обозначении прямой суммы представлений будем для удобства пользоваться знаком <<$+$>> вместо знака <<$\oplus$>>.

Обозначим через~$V$ комплексное пространство $V_{\R}(\Cbb)$, а~через~$\ggt$\т комплексную редуктивную линейную алгебру Ли $\ggt_{\R}(\Cbb)\subs\br{\slg(V_{\R})}(\Cbb)=\slg(V)$, имеющую согласно предположению простой коммутант. Для некоторого точного неприводимого представления~$R$ алгебры Ли~$\ggt$ тавтологическое представление $\ggt\cln V$ совпадает с~одним из представлений $R$ и~$R+R'$. Таким образом, мы можем говорить о~неприводимой линейной алгебре Ли $R(\ggt)$, изоморфной алгебре Ли~$\ggt$ как абстрактная алгебра, и~о~неприводимой простой линейной алгебре Ли $R\br{[\ggt,\ggt]}$.

Если линейная алгебра Ли $\ggt\subs\slg(V)$ приводима, то пространство~$V_{\R}$ обладает $\ggt_{\R}$\д инвариантной комплексной структурой, в~результате чего естественным образом возникает комплексная линейная алгебра Ли
$\ggt=\ggt_{\R}(\Cbb)=\ggt_{\R}\oplus i\ggt_{\R}\subs\glg_{\Cbb}(V_{\R})$.

Говоря о~неразложимых системах простых корней, мы будем использовать стандартную нумерацию простых корней (см., например, \cite[табл.\,1]{VO}, \cite[табл.\,1]{Elsh}) и~обозначать $i$\д й базисный вес через~$\ph_i$.

В~настоящей работе будут доказаны теоремы \ref{main}---\ref{main2}.

\begin{theorem}\label{main} Допустим, что условие~\eqref{elem} не выполняется, а~\ter{простая} линейная алгебра Ли $R\br{[\ggt,\ggt]}$ не является ни присоединённой, ни классической.
\begin{nums}{1}
\renewcommand{\labhi}[1]{\labheadu{\arabic{#1}}}
\renewcommand{\labhii}[1]{\labheadi{\arabic{#1}}}
\item Если линейная алгебра $\ggt\subs\slg(V)$ неприводима, то она совпадает с~одной из следующих линейных алгебр\:
\begin{nums}{-1}
\item $\ph_4(A_7)$, $(\ph_1+\ph_2+\ph_3)(A_3)$\~
\item $\ph_r(B_r)$, $(\ph_1+\ph_r)(B_r)$, $(\ph_2+\ph_r)(B_r)$ \ter{$r=3,4$}\~
\item $\ph_3(B_4)$, $(\ph_1+\ph_2+\ph_3)(B_3)$\~
\item $\ph_2(C_r)$ \ter{$r>2$}\~
\item $(\ph_1+\ph_3)(C_3)$, $(\ph_1+\ph_2+\ph_3)(C_3)$, $\ph_4(C_4)$, $(\ph_2+\ph_4)(C_4)$, $\ph_4(C_5)$\~
\item $\ph_8(D_8)$\~
\item $(\ph_1+\ph_r)(D_r)$ \ter{$r=4,8$}\~
\item $\ph_2(E_r)$ \ter{$r=7,8$}\~
\item $(\ph_1+\ph_{r-1})(E_r)$ \ter{$r=6,8$}\~
\item $\ph_7(E_8)$\~
\item $\ph_1(F_4)$, $\ph_2(F_4)$, $\ph_3(F_4)$, $(\ph_1+\ph_3)(F_4)$, $(\ph_1+\ph_4)(F_4)$, $(\ph_2+\ph_4)(F_4)$\~
\item $\ph_1(G_2)$.
\end{nums}
\item Если линейная алгебра $\ggt\subs\slg(V)$ приводима, то комплексная линейная алгебра Ли $\ggt\subs\glg_{\Cbb}(V_{\R})$ совпадает с~одной из
следующих линейных алгебр\:
\begin{nums}{-1}
\item $\ph_2(A_r)$ \ter{$r>3$}\~
\item $(\ph_1+\ph_{r-1})(A_r)$ \ter{$r>2$}\~
\item $\ph_3(A_r)$ \ter{$r=5,6,7$}\~
\item $\ph_r(B_r)$, $(\ph_1+\ph_r)(B_r)$, $(\ph_2+\ph_r)(B_r)$ \ter{$r=5,6$}\~
\item $\ph_3(C_3)$, $(\ph_2+\ph_3)(C_3)$, $\ph_3(C_4)$, $(\ph_1+\ph_4)(C_4)$\~
\item $\ph_r(D_r)$, $(\ph_1+\ph_r)(D_r)$ \ter{$r=5,6,7$}\~
\item $\ph_1(E_6)$, $\ph_2(E_6)$, $\ph_1(E_7)$, $\ph_7(E_7)$, $(\ph_1+\ph_6)(E_7)$\~
\item $\ph_2(A_{r-1})\oplus\Cbb E$ \ter{$r>3$}\~
\item $\ph_3(A_{r-1})\oplus\Cbb E$ \ter{$r=6,7$}\~
\item $\ph_{r-1}(B_{r-1})\oplus\Cbb E$ \ter{$r=4,5,6$}\~
\item $(\ph_1+\ph_{r-1})(B_{r-1})\oplus\Cbb E$ \ter{$r=3\sco6$}\~
\item $\ph_2(C_{r-1})\oplus\Cbb E$ \ter{$r>3$}\~
\item $\ph_3(C_{r-1})\oplus\Cbb E$ \ter{$r=4\sco7$}\~
\item $(\ph_1+\ph_{r-1})(C_{r-1})\oplus\Cbb E$ \ter{$r=3,4$}\~
\item $\ph_{r-1}(D_{r-1})\oplus\Cbb E$ \ter{$r=6,7,8$}\~
\item $\ph_1(E_{r-1})\oplus\Cbb E$ \ter{$r=7,8$}\~
\item $\ph_1(F_4)\oplus\Cbb E$, $\ph_2(F_4)\oplus\Cbb E$, $(\ph_1+\ph_4)(F_4)\oplus\Cbb E$\~
\item $\ph_1(G_2)\oplus\Cbb E$, $(\ph_1+\ph_2)(G_2)\oplus\Cbb E$.
\end{nums}
\end{nums}
\end{theorem}

\begin{note} Если \ter{простая} линейная алгебра Ли $R\br{[\ggt,\ggt]}$ является присоединённой, то возможны следующие случаи\:
\begin{nums}{-1}
\item линейная алгебра $\ggt_{\R}\subs\slg(V_{\R})$ является присоединённой простой компактной линейной алгеброй Ли\~
\item линейная алгебра $\ggt\subs\slg(V)$ приводима, а~комплексная линейная алгебра $\ggt\subs\glg_{\Cbb}(V_{\R})$ совпадает с~линейной алгеброй
$\ad\br{[\ggt,\ggt]}\oplus\Cbb E\subs\glg\br{[\ggt,\ggt]}$.
\end{nums}
\end{note}

\begin{theorem}\label{main1} Если линейная алгебра $\ggt_{\R}\subs\slg(V_{\R})$ является присоединённой простой компактной линейной алгеброй Ли, то условие~\eqref{elem} не выполняется. Если линейная алгебра $\ggt\subs\slg(V)$ приводима, а~линейная алгебра $\ggt\subs\glg_{\Cbb}(V_{\R})$ совпадает с~линейной алгеброй $\ad\br{[\ggt,\ggt]}\oplus\Cbb E\subs\glg\br{[\ggt,\ggt]}$, то условие~\eqref{elem} выполняется.
\end{theorem}

\begin{theorem}\label{main2} Предположим, что \ter{простая} линейная алгебра Ли $R\br{[\ggt,\ggt]}$ является классической. Условие~\eqref{elem} выполняется тогда и~только тогда, когда линейная алгебра $\ggt\subs\slg(V)$ приводима, а~комплексная линейная алгебра Ли $\ggt\subs\glg_{\Cbb}(V_{\R})$ совпадает с~одной из линейных алгебр $\sog_{2r-1}\oplus\Cbb E$ \ter{$r=2,3$}, $\slg_3$, $\glg_2$ и~$\spg_4$.
\end{theorem}

В~\Ss\ref{facts} введён ряд используемых в~работе обозначений, а~также доказаны некоторые вспомогательные утверждения. В~\Ss\ref{promain} приведено доказательство теорем \ref{main}---\ref{main2}. Наконец, в~\Ss\ref{tables} размещены все необходимые таблицы.

Автор благодарит профессора Э.\,Б.\?Винберга за многолетнее научное руководство, постоянную поддержку и~ценные советы.

\section{Обозначения и вспомогательные факты}\label{facts}

В~этом параграфе все пространства (в~частности, алгебры Ли и~их представления) рассматриваются, если не оговорено противное, над полем~$\Cbb$.

Если $\ggt_i\subs\slg(V_i)$ ($i=1\sco m$)\т произвольные линейные алгебры Ли, то представление~$R$ алгебры Ли $\ggt:=\ggt_1\sop\ggt_m$ в~пространстве $V:=V_1\sot V_m$, служащее тензорным произведением тавтологических представлений $\ggt_i\cln V_i$ ($i=1\sco m$), является точным, и~мы будем называть линейную алгебру Ли $R(\ggt)\subs\slg(V)$, изоморфную алгебре Ли~$\ggt$ как абстрактная алгебра, \textit{тензорным произведением линейных алгебр Ли $\ggt_i\subs\slg(V_i)$, $i=1\sco m$}.

Хорошо известно, что любая неприводимая полупростая линейная алгебра Ли единственным образом представляется в~виде тензорного произведения
неприводимых простых линейных алгебр Ли, называемых её \textit{простыми компонентами}.

Упорядоченную пару $(\Pi;\al)$, где $\Pi$\т неразложимая система простых корней, а~$\al$\т её корень с~номером~$i$, будем называть \textit{системой с~отмеченным корнем} и~обозначать через $\Pi(\al)$ или через $\Pi(i)$. Назовём \textit{изоморфизмом} двух систем с~отмеченным корнем изоморфизм соответствующих неразложимых систем простых корней, переводящий отмеченные корни друг в~друга. Будем говорить, что две системы с~отмеченным корнем
\textit{изоморфны}, если между ними существует изоморфизм, и~рассматривать все системы с~отмеченным корнем с~точностью до изоморфизма.

Обозначим через~$\Om$ множество следующих систем с~отмеченным корнем\:
\begin{nums}{-1}
\item $\Pi(1)$, $\Pi\ne E_8$\~
\item $A_r(2)$, $r>2$\~
\item $A_5(3)$\~
\item $B_r(r)$, $r=3,4$\~
\item $C_r(2)$, $r>2$\~
\item $D_r(r)$, $r=5,6$.
\end{nums}
Для неразложимой системы простых корней~$\Pi$ положим $\pd\Pi:=\bc{\al\in\Pi\cln\Pi(\al)\in\Om}\subs\Pi$, а~также
$\Inn{\Pi}:=\Pi\sm(\pd\Pi)=\bc{\al\in\Pi\cln\Pi(\al)\notin\Om}\subs\Pi$. Все неразложимые системы простых корней~$\Pi$, для которых $\pd\Pi\ne\Pi$ (или, что равносильно, $\Inn{\Pi}\ne\es$), перечислены в~таблице~\ref{tablId} с~указанием подмножеств $\pd\Pi\subs\Pi$ и~$\Inn{\Pi}\subs\Pi$.

Фундаментальные неприводимые представления произвольной простой алгебры Ли находятся в~естественной биекции с~корнями её (неразложимой) системы простых корней. Это позволяет установить биекцию между всеми (с~точностью до изоморфизма) неприводимыми простыми линейными алгебрами Ли, тавтологическое представление которых является фундаментальным, и~всеми системами с~отмеченным корнем. В~дальнейшем мы будем отождествлять каждую из указанных линейных алгебр Ли с~соответствующей ей системой с~отмеченным корнем. Таким образом, множество~$\Om$ можно понимать как определённый класс неприводимых простых линейных алгебр Ли. Обозначим через~$\wt{\Om}$ класс всех простых линейных алгебр Ли множества~$\Om$, имеющих ранг более~$1$.

Пусть $\ggt\subs\glg(V)$\т неприводимая редуктивная линейная алгебра Ли. Тогда алгебра $\ggt\subs\glg(V)$ разлагается в~прямую сумму своего центра и~коммутанта, содержащихся в~$\Cbb E$ и~$\slg(V)$ соответственно. Обозначим через~$R$ тавтологическое представление $\ggt\cln V$, а~через~$\wt{R}$\т представление алгебры~$\ggt$, равное $R$ (соотв. $R+R'$), если представление~$R$ ортогонально (соотв. неортогонально).

Пусть $\ggt_i\subs\slg(V_i)$ ($i=1\sco m$)\т простые компоненты неприводимой полупростой линейной алгебры Ли $[\ggt,\ggt]\subs\slg(V)$. Положим $r:=\rk\ggt$ и~$r_i:=\rk\ggt_i$ ($i=1\sco m$).

\begin{lemma}\label{ome} Допустим, что $m>0$, $r_1\ge r-1$, а~с.п.о.п. представления~$\wt{R}$ алгебры~$\ggt$ нетривиальна. Тогда либо $\ggt\subs\glg(V)$ есть присоединённая простая линейная алгебра Ли, либо линейная алгебра Ли $\ggt_1\subs\slg(V_1)$ относится к~классу~$\Om$.
\end{lemma}

\begin{proof} Рассмотрим произвольное число $i=1\sco m$. Обозначим через~$R_i$ тавтологическое представление $\ggt_i\cln V_i$. Положим
$n_i:=\dim V_i$. Далее, обозначим через $l_i$ (соотв.~$L_i$) индекс простой линейной алгебры Ли~$\ggt_i$ (соотв. простой линейной алгебры Ли~$\ggt_i$ как подалгебры линейной алгебры Ли $[\ggt,\ggt]\subs\slg(V)$). Всякий раз, говоря об индексе простой компоненты~$\ggt_i$ неприводимой полупростой линейной алгебры Ли $[\ggt,\ggt]\subs\slg(V)$, мы будем подразумевать число~$L_i$.

Очевидно, что для любого $i=1\sco m$ число~$L_i$ равно произведению числа~$l_i$ и~всех чисел $n_j$, $j\in\{1\sco m\}\sm\{i\}$.

Возможны следующие случаи.

\begin{cass}{sim} Алгебра~$\ggt$ проста\т что равносильно, $\ggt_1=\ggt$.
\end{cass}

\begin{cass}{nlt} Центр алгебры~$\ggt$ нетривиален, а~линейная алгебра~$\ggt_1$ не является локально транзитивной.
\end{cass}

\begin{cass}{lt} Линейная алгебра~$\ggt_1$ локально транзитивна.
\end{cass}

\begin{cass}{du} Линейная алгебра~$\ggt$ является полупростой, но не является простой и~обладает по крайней мере одной классической простой компонентой индекса менее~$1$.
\end{cass}

\begin{cass}{ndu} Линейная алгебра~$\ggt$ является полупростой, но не является простой, а~все её классические простые компоненты имеют индекс не менее~$1$.
\end{cass}

Предположим, что имеет место случай~\ref{nlt}. Тогда $m=1$, $r_1=r-1$, $\ggt_1=[\ggt,\ggt]$, $V_1=V$ и~$\ggt=\Cbb E\oplus\ggt_1\subs\glg(V)$. Значит, представление $\ggt\cln V$ неортогонально. Поскольку линейная алгебра~$\ggt_1$ не является локально транзитивной, с.п.о.п. указанного представления содержится в~подалгебре $\ggt_1\subs\ggt$\~ то же можно сказать и~о~с.п.о.п. представления~$\wt{R}$.

Таким образом, в~каждом из случаев \ref{sim} и~\ref{nlt} представление $\wt{R}|_{\ggt_1}$ алгебры~$\ggt_1$ обладает нетривиальной с.п.о.п. и~совпадает с~представлением $R_1$ либо $R_1+R'_1$, причём первый вариант возможен только в~случае~\ref{sim} при условии ортогональности представления
$\ggt\cln V$.

Все простые линейные алгебры Ли с~нетривиальной с.п.о.п. суть в~точности все присоединённые простые линейные алгебры Ли, а~также все линейные алгебры Ли, перечисленные в~\cite[табл.\,1,\,2]{Elsh} и~\cite[табл.\,0]{AMP}. Выбрав среди тавтологических представлений указанных линейных алгебр Ли ортогональные неприводимые представления и~прямые суммы двух сопряжённых друг другу неприводимых представлений (учитывая при этом неприводимость присоединённого представления произвольной простой алгебры Ли), получаем требуемое утверждение в~каждом из случаев \ref{sim} и~\ref{nlt}.

Теперь разберём случай~\ref{lt}. Согласно результатам работы~\cite{Shp}, любая локально транзитивная неприводимая простая линейная алгебра Ли
может быть получена из линейной алгебры $A_r(1)$, $A_{2r}(2)$, $C_r(1)$ ($r\in\N$) либо $D_5(5)$ при помощи принципа двойственности (см.~\cite{Shp}). Применение последнего к~неприводимой полупростой линейной алгебре Ли не влияет на её простые компоненты, отличные от линейных алгебр $A_r(1)$, и~таким образом лемма в~случае~\ref{lt} доказана.

Далее будем предполагать, что имеет место один из случаев \ref{du} и~\ref{ndu}, а~линейная алгебра~$\ggt_1$ не относится к~классу~$\Om$.

Справедливы следующие утверждения\:
\begin{nums}{-1}
\item $m=2$, $r_1=r-1$, $r_2=1$ и~$\ggt\subs\slg(V)$\~
\item если линейная алгебра Ли~$\ggt_1$ является классической, то $\ggt_1=\sog_3$\~
\item если $n_1<4$, то $\ggt_1=\sog_3$\~
\item если линейная алгебра Ли~$\ggt_2$ является классической, то $\ggt_2=\slg_2$ либо $\ggt_2=\sog_3$\~
\item ни одна из простых компонент неприводимой полупростой линейной алгебры Ли~$\ggt$ не относится к~классу~$\wt{\Om}$.
\end{nums}

Если $\ggt_i=\sog_3$ для некоторого $i=1,2$, то $L_i\ge l_i=1$, а~если $\ggt_2=\slg_2$, то $l_2=\frac{1}{4}$ и~$L_2=\frac{n_1}{4}$. Значит, в~случае~\ref{du} неприводимая полупростая линейная алгебра~$\ggt$ имеет две простых компоненты $\ggt_1=\sog_3$ и~$\ggt_2=\slg_2$. Для такой линейной алгебры~$\ggt$ представление~$R$ симплектично, его с.п.о.п. одномерна, представление~$R'$ точно, и, как следствие, с.п.о.п. представления $\wt{R}=R+R'$ тривиальна, что приводит нас в~случае~\ref{du} к~противоречию.

Все неприводимые полупростые линейные алгебры Ли, не являющиеся простыми, не имеющие классических простых компонент индекса менее~$1$ и~обладающие нетривиальной с.п.о.п., перечислены в~\cite[табл.\,5,\,6]{Elash}. В~этих таблицах каждая линейная алгебра Ли представлена в~виде тензорного произведения неприводимых полупростых линейных алгебр Ли $\ggt'_i\subs\slg(V'_i)$, $i=1\sco m'$ (в~указанном порядке). Для всех линейных алгебр таблицы~6, а~также для линейных алгебр №№\,4,\,7,\,8 таблицы~5 линейная алгебра~$\ggt'_1$ относится к~классу~$\wt{\Om}$. Кроме того, линейные алгебры №№\,2,\,3,\,6 таблицы~5 имеют по три простых компоненты. Значит, в~случае~\ref{ndu} линейная алгебра~$\ggt$ совпадает с~одной из линейных алгебр №№\,1,\,5 таблицы~5, т.\,е. $m'=2$, $\ggt'_1=\sog_{n_1}$ и~$n_1=n_2+2$. При этом если $n_1=3$, то $n_2=1$, $\ggt'_1=\sog_3$, $\ggt'_2=0$, $m=1$\~ если $n_1=4$, то $n_2=2$, $\ggt'_1=\sog_4\cong\slg_2\oplus\slg_2$ и~$\ggt'_2=\slg_2$, откуда $m=3$\~ если же $n_1>4$, то линейная алгебра $\ggt'_1=\sog_{n_1}$ относится к~классу~$\wt{\Om}$. Тем самым мы и~в~случае~\ref{ndu} пришли к~противоречию.

Теперь лемма полностью доказана.
\end{proof}

\begin{lemma}\label{sl2} Допустим, что $r=1$, а~с.п.о.п. представления~$\wt{R}$ алгебры~$\ggt$ нетривиальна. Тогда $\ggt\subs\glg(V)$ есть присоединённая простая линейная алгебра Ли $\ad(\slg_2)$.
\end{lemma}

\begin{proof} Имеем $\rk\ggt=r=1<\dim\ggt$, поскольку с.п.о.п. точного представления~$\wt{R}$ алгебры~$\ggt$ нетривиальна. Значит, $m=1$, $r_1=1=r>r-1$, $\ggt\cong\slg_2$, $V_1=V$, $\ggt_1=\ggt\subs\slg(V)$. Предположим, что линейная алгебра $\ggt\subs\slg(V)$ не является присоединённой. Согласно лемме~\ref{ome}, линейная алгебра $\ggt\subs\slg(V)$ относится к~классу~$\Om$ и, ввиду равенства $r=1$, совпадает с~линейной алгеброй $A_1(1)=\slg_2$. Тавтологическое представление~$R$ алгебры $\ggt=\slg_2$ симплектично, а~с.п.о.п. её представления $\wt{R}=R+R'=R+R$ тривиальна, что противоречит условию.
\end{proof}

\begin{stm} Пусть $\Pi$ и~$\Pi'$\т неразложимые системы простых корней, такие что $\Pi'\subs\Pi$ и~$|\Pi'|=|\Pi|-1$, а~$\al\in\Pi'$\т некоторый корень. Если $\Pi(\al)\in\Om$, то $\Pi'(\al)\in\Om$.
\end{stm}

\begin{proof} Если $|\Pi'|=1$, то $\Pi'(\al)=A_1(1)\in\Om$. Далее будем считать, что $|\Pi|>2$.

Согласно условию, $\Pi'=\Pi\sm\{\be\}$, где $\be\in\Pi\sm\{\al\}$\т корень неразложимой системы простых корней~$\Pi$, соответствующий висячей вершине её (неразложимой) схемы Дынкина. Всевозможные системы с~отмеченным корнем $\Pi'(\al)$, получаемые таким образом при $\Pi(\al)\in\Om$ и~$|\Pi|>2$,
перечислены в~таблице~\ref{tablOm} (где всюду предполагается, что $r>2$). Осталось воспользоваться таблицей~\ref{tablId}.
\end{proof}

\begin{imp}\label{ssh} Пусть $\Pi$ и~$\Pi'$\т неразложимые системы простых корней, такие что $\Pi'\subs\Pi$ и~$|\Pi'|\ge|\Pi|-1$. Если для корня $\al\in\Pi'$ имеем $\Pi(\al)\in\Om$, то $\Pi'(\al)\in\Om$.
\end{imp}

Следующие два утверждения ввиду их известности приводятся без доказательства.

\begin{stm}\label{fai} Неприводимое представление произвольной комплексной редуктивной алгебры Ли~$\ggt$ является точным тогда и~только тогда, когда его веса относительно картановской подалгебры $\tgt\subs\ggt$ линейно порождают пространство~$\tgt^*$.
\end{stm}

\begin{stm}\label{lind} Пусть $\Eb$\т евклидово пространство, $\Pi\subs\Eb$\т неразложимая система простых корней, а~$C\subs\Eb$\т её камера Вейля. Тогда для любого подмножества $\Pi'\subs\Pi$, отличного от~$\Pi$, имеем $\ha{\Pi'}_{\R}\cap C=\{0\}$.
\end{stm}

\section{Доказательства результатов}\label{promain}

Данный параграф посвящён доказательству теорем \ref{main}---\ref{main2}.

Вернёмся к~обозначениям и~предположениям из \Ss\ref{introd}.

Начнём с~доказательства теоремы~\ref{main}.

Фиксируем максимальную коммутативную подалгебру~$\tgt_{\R}$ алгебры~$\ggt_{\R}$ и~картановскую подалгебру $\tgt:=\tgt_{\R}(\Cbb)$ алгебры~$\ggt$.
В~результате возникают система корней $\De\subs\tgt^*$, группа Вейля $W\subs\GL(\tgt^*)$, система положительных корней $\De^+\subs\De$ и~система простых корней $\Pi\subs\De^+\subs\De$. Поскольку алгебра $[\ggt,\ggt]$ проста, система простых корней $\Pi\subs\tgt^*$ неразложима. Обозначим через~$\Pc$ семейство всех неразложимых систем простых корней $\Pi'\subs\Pi\subs\tgt^*$ порядка $r-2$. В~системе простых корней~$\Pi$ подмножество всех корней, не принадлежащих ни одному из подмножеств $\Inn{\Pi}{}'$ ($\Pi'\in\Pc$), обозначим через $\pd_r\Pi$.

Пусть $\La\subs\tgt^*$\т система весов представления~$R$, а~$\la\in\La$\т его старший вес относительно системы простых корней $\Pi\subs\De$. Положим $\Pi_{\la}:=\bc{\al\in\Pi\cln\ha{\la|\al}\ne0}\subs\Pi$.

Имеем $\ha{W\la}=\ha{\La}=\tgt^*$, $\ba{\{\la\}\cup\Pi}=\tgt^*$, $r=\dim\tgt^*\in\bc{|\Pi|;|\Pi|+1}$, $|\Pi|\in\{r;r-1\}$.
В~частности, $r-2<|\Pi|$.

\begin{stm}\label{cov} Предположим, что $r>2$. Тогда система простых корней~$\Pi$ совпадает с~объединением всех своих подмножеств $\Pi'\in\Pc$.
\end{stm}

\begin{proof} Имеем $0<r-2<|\Pi|$. Значит, в~(неразложимой) схеме Дынкина системы простых корней~$\Pi$ для любой вершины найдётся содержащая её неразложимая подсхема порядка $r-2$.
\end{proof}

Пусть $\Pi'\in\Pc$\т некоторая система простых корней, удовлетворяющая условию
\eqn{\label{nes}
(r>2)\quad\Ra\quad(\Pi_{\la}\cap\Pi'\ne\es),}
а~$W'\subs\GL(\tgt^*)$\т её группа Вейля.

Имеем $|\Pi'|=r-2<|\Pi|$. Согласно утверждению~\ref{lind}, система $\br{\{\la\}\cup\Pi'}\subs\tgt^*$ состоит из $r-1$ линейно независимых линейных функций, принимающих на вещественной форме~$\tgt_{\R}$ пространства~$\tgt$ чисто мнимые значения. Следовательно, пересечение ядер этих линейных функций имеет вид $\Cbb\xi\subs\tgt$, $\xi\in\tgt_{\R}\sm\{0\}$. В~пространстве~$\tgt^*$ подпространство всех линейных функций, обращающихся в~нуль на векторе $\xi\in\tgt_{\R}$, можно отождествить с~пространством~$\tgt_0^*$, где $\tgt_0:=\tgt/(\Cbb\xi)$. Ясно, что $\dim\tgt_0=r-1$ и~что $\tgt_0^*=\ba{\{\la\}\cup\Pi'}\subs\tgt^*$.

Докажем, что
\eqn{\label{span}
\Pi'\subs\ha{W'\la}.}

Для любого корня $\al\in\Pi_{\la}\cap\Pi'$ имеем $\al=\frac{1}{\ha{\la|\al}}(\la-r_{\al}\la)\in\ha{W'\la}$, $\al\in\ha{W'\la}\cap\ha{\Pi'}$. Если $r>2$, то $\Pi_{\la}\cap\Pi'\ne\es$, $\ha{W'\la}\cap\ha{\Pi'}\ne0$, что влечёт~\eqref{span}, поскольку $W'$\д подмодуль $\ha{\Pi'}\subs\tgt^*$ прост. Если же $r=2$, то $|\Pi'|=r-2=0$, $\Pi'=\es\subs\ha{W'\la}$.

Согласно~\eqref{span}, $\tgt_0^*=\ba{\{\la\}\cup\Pi'}=\ha{W'\la}\subs\tgt^*$.

Обозначим через~$\De_0$ систему корней $(\De\cap\tgt_0^*)\subs\tgt_0^*$, а~через~$\De^+_0$\т систему положительных корней $(\De^+\cap\tgt_0^*)\subs\De_0$. Последней соответствует система простых корней $\Pi_0\subs\De_0$. Имеем $\Pi'\subs\Pi\subs\De^+$ и~$\Pi'\subs\tgt_0^*$, откуда $\Pi'\subs\De^+_0\subs\De_0$. Кроме того, $\Pi\cap(\De^++\De^+)=\es$, $\Pi'\cap(\De^+_0+\De^+_0)=\es$, и~поэтому
$\Pi'\subs\Pi_0$.

В~алгебре~$\ggt$ централизатор~$\ggt'$ элемента $\xi\in\tgt_{\R}$ согласован с~её компактной вещественной формой~$\ggt_{\R}$ и~содержит картановскую подалгебру $\tgt\subs\ggt$. Значит, $\ggt'$\т комплексная редуктивная алгебра с~компактной вещественной формой $\ggt'_{\R}:=\ggt'\cap\ggt_{\R}$, картановской подалгеброй $\tgt\subs\ggt$ и~системой корней $\De_0\subs\tgt^*$, а~$\Cbb\xi$\т центральный идеал алгебры~$\ggt'$, согласованный с~её компактной вещественной формой~$\ggt'_{\R}$. В~свою очередь $\ggt_0:=\ggt'/(\Cbb\xi)$\т комплексная редуктивная алгебра Ли с~компактной вещественной формой $\ggt'_{\R}/(\R\xi)$, картановской подалгеброй~$\tgt_0$ и~системой корней $\De_0\subs\tgt_0^*$, причём $\rk\ggt_0=\dim\tgt_0=r-1$.

Пусть $V_0\subs V$\т ядро оператора $\xi\in\tgt_{\R}\subs\ggt\subs\slg(V)$. Ограничение тавтологического представления $\ggt'\cln V$ на (очевидно, инвариантное) подпространство $V_0\subs V$ естественным образом индуцирует представление $\ggt_0\cln V_0$, поскольку $(\Cbb\xi)V_0=0$. При этом
\eqn{\label{refo}
\begin{split}
V_0=(V_0\cap V_{\R})\oplus i(V_0\cap V_{\R})=(V_0\cap V_{\R})(\Cbb);\\
\br{\ggt'_{\R}/(\R\xi)}(V_0\cap V_{\R})=\ggt'_{\R}(V_0\cap V_{\R})\subs(V_0\cap V_{\R}).
\end{split}}

Алгебра~$\ggt_0$ обладает неприводимым представлением~$R_0$ со старшим весом $\la\in\tgt_0^*$ относительно системы положительных корней $\De^+_0\subs\De_0$. Если $\pi\cln\ggt'\thra\ggt_0$\т канонический эпиморфизм алгебр, то $R_0\circ\pi$\т неприводимое представление алгебры~$\ggt'$ со старшим весом $\la\in\tgt^*$ относительно системы положительных корней $\De^+_0\subs\De_0$. Обозначим через~$\wt{R}_0$ представление алгебры~$\ggt_0$, равное $R_0$ (соотв. $R_0+R'_0$), если представление~$R_0$ ортогонально (соотв. неортогонально).

\begin{prop} Представление~$R_0$ алгебры~$\ggt_0$ точно.
\end{prop}

\begin{proof} Пусть $\La_0\subs\tgt_0^*$\т система весов представления~$R_0$ алгебры~$\ggt_0$. Имеем $\Pi'\subs\De_0$,
$W'\la\subs\La_0$, $\tgt_0^*=\ha{W'\la}\subs\ha{\La_0}\subs\tgt_0^*$, $\ha{\La_0}=\tgt_0^*$. Осталось применить утверждение~\ref{fai}.
\end{proof}

Таким образом, мы можем отождествить алгебру~$\ggt_0$ с~неприводимой редуктивной линейной алгеброй $R_0(\ggt_0)$ и~говорить о~простых компонентах её коммутанта\т простых линейных алгебрах Ли $\ggt_i\subs\slg(V_i)$ ($i=1\sco m$)\т и~о~соответствующих им (в~указанном порядке) неразложимых компонентах $\Pi_1\sco\Pi_m$ системы простых корней $\Pi_0\subs\De_0$.

Если $\ggt_0$ есть присоединённая линейная алгебра, то $\la\in\tgt_0^*$\т старший корень системы корней $\De_0\subs\tgt_0^*$, откуда $\la\in\De_0\subs\De\subs\tgt^*$.

\begin{lemma}\label{r2} Допустим, что $r=2$, а~с.п.о.п. представления~$\wt{R}_0$ алгебры~$\ggt_0$ нетривиальна. Тогда $\la\in\De$.
\end{lemma}

\begin{proof} Достаточно применить лемму~\ref{sl2} к~линейной алгебре~$\ggt_0$ ранга $r-1=1$.
\end{proof}

\begin{lemma}\label{lar} Допустим, что $r>2$, $\la\notin\De$, а~с.п.о.п. представления~$\wt{R}_0$ алгебры~$\ggt_0$ нетривиальна. Тогда подмножество
$(\Pi_{\la}\cap\Pi')\subs\Pi$ включает в~себя единственный корень $\al\in\Pi$. При этом
\begin{nums}{-1}
\item $\ha{\la|\al}=1$ и~$\Pi'(\al)\in\Om$\~
\item если система простых корней~$\Pi_0$ неразложима, то $\Pi_0(\al)\in\Om$, а~также $\ha{\la|\be}=0$ для всякого $\be\in\Pi_0\sm\{\al\}$.
\end{nums}
\end{lemma}

\begin{proof} Неразложимая система простых корней $\Pi'\subs\tgt_0^*$ порядка $r-2>0$ содержится в~системе простых корней $\Pi_0\subs\tgt_0^*$, а~значит, и~в~некоторой её неразложимой компоненте $\Pi_i$, $i=1\sco m$ (в~частности, $m>0$). Не умаляя общности, будем считать, что $\Pi'\subs\Pi_1$. Имеем
\begin{gather}
|\Pi_1|\le|\Pi_0|\le\dim\tgt_0^*=r-1=|\Pi'|+1;\label{tpi}\\
\rk\ggt_1=|\Pi_1|\ge|\Pi'|=r-2=\rk\ggt_0-1.\label{gpi}
\end{gather}
Поскольку $r>2$, подмножество $(\Pi_{\la}\cap\Pi')\subs\Pi$ не является пустым и~включает в~себя некоторый корень $\al\in\Pi$. Применяя лемму~\ref{ome} к~линейной алгебре~$\ggt_0$ и~учитывая~\eqref{gpi}, получаем, что линейная алгебра $\ggt_1\subs\slg(V_1)$ относится к~классу~$\Om$. Это означает, что
\begin{nums}{-1}
\item $\ha{\la|\al}=1$\~
\item $\ha{\la|\be}=0$ для всякого $\be\in\Pi_1\sm\{\al\}$ (в~частности, $\ha{\la|\be}=0$ для всякого $\be\in\Pi'\sm\{\al\}$)\~
\item $\Pi_1(\al)\in\Om$ (и, согласно соотношению~\eqref{tpi} и~следствию~\ref{ssh}, $\Pi'(\al)\in\Om$).\qedhere
\end{nums}
\end{proof}

\begin{lemma}\label{tr} Предположим, что с.п.о.п. представления~$\wt{R}_0$ алгебры~$\ggt_0$ тривиальна. Тогда найдётся вектор $v\in V_{\R}$, такой что $\xi\in(\ggt_{\R})_v$ и~$\rk\br{(\ggt_{\R})_v}=1$.
\end{lemma}

\begin{proof} Напомним, что $\la\in\tgt_0^*\subs\tgt^*$\т старший вес представления~$R$ алгебры~$\ggt$ относительно системы простых корней $\Pi\subs\De$. В~силу~\eqref{refo}, представление~$\wt{R}_0$ алгебры~$\ggt_0$ вкладывается в~представление $\ggt_0\cln V_0$. Значит, с.п.о.п. последнего тривиальна. Вновь пользуясь соотношениями~\eqref{refo}, получаем, что с.п.о.п. представления $\br{\ggt'_{\R}/(\R\xi)}\cln(V_0\cap V_{\R})$
тривиальна. Следовательно, существует вектор $v\in V_0\cap V_{\R}$, для которого $\ggt_v\cap\ggt'_{\R}=\R\xi$. Итак,
$\br{(\ggt_{\R})_v}\cap\ggt'=\ggt_v\cap\ggt'_{\R}=\R\xi$, т.\,е. подалгебра $(\ggt_{\R})_v\subs\ggt_{\R}$ содержит $\R\xi$ в~качестве максимальной коммутативной подалгебры и~потому имеет ранг~$1$.
\end{proof}

\begin{note} Схема доказательства леммы~\ref{tr} унаследована из~\cite{CPV}, где используется аналогичный метод нахождения точек, имеющих замкнутую орбиту и~стабилизатор ранга~$1$, для комплексных редуктивных линейных групп Ли.
\end{note}

Из лемм \ref{r2}---\ref{tr} вытекает следующая лемма.

\begin{lemma}\label{PPi} Допустим, что $\la\notin\De$, а~в~пространстве~$V_{\R}$ не существует вектора~$v$, такого что $\xi\in(\ggt_{\R})_v$ и~$\rk\br{(\ggt_{\R})_v}=1$. Тогда $r>2$ и~$\Pi_{\la}\cap\Pi'=\{\al\}\subs\Pi$ \ter{$\al\in\Pi$}. При этом
\begin{nums}{-1}
\item $\ha{\la|\al}=1$ и~$\al\in\pd\Pi'$\~
\item если система простых корней~$\Pi_0$ неразложима, то $\al\in\pd\Pi_0$, а~также $\ha{\la|\be}=0$ для всякого $\be\in\Pi_0\sm\{\al\}$.
\end{nums}
\end{lemma}

До сих пор предполагалось, что \textit{фиксирована} система простых корней $\Pi'\in\Pc$, удовлетворяющая~\eqref{nes}. В~дальнейшем же мы будем всякий раз выбирать её в~зависимости от конкретной ситуации. В~таком случае из леммы~\ref{PPi} вытекает следующая лемма.

\begin{lemma}\label{PP} Предположим, что $\la\notin\De$, а~условие~\eqref{elem} не выполняется. Пусть $\Pi'\in\Pc$\т система простых корней, удовлетворяющая~\eqref{nes}. В~пространстве~$\tgt^*$ обозначим через~$\Pi_0$ систему простых корней, соответствующую системе положительных корней $\De^+\cap\ba{\{\la\}\cup\Pi'}$. Тогда $r>2$ и~$\Pi_{\la}\cap\Pi'=\{\al\}\subs\Pi$ \ter{$\al\in\Pi$}. При этом
\begin{nums}{-1}
\item $\ha{\la|\al}=1$ и~$\al\in\pd\Pi'$\~
\item если система простых корней~$\Pi_0$ неразложима, то $\al\in\pd\Pi_0$, а~также $\ha{\la|\be}=0$ для всякого $\be\in\Pi_0\sm\{\al\}$.
\end{nums}
\end{lemma}

\begin{lemma}\label{faP} Допустим, что $\la\notin\De$, а~условие~\eqref{elem} не выполняется. Тогда $r>2$, а~для любой системы простых корней $\Pi'\in\Pc$ имеем
\eqn{\label{cPi}\begin{aligned}
&|\Pi_{\la}\cap\Pi'|\le1;\quad\quad&\Pi_{\la}\cap\Inn{\Pi}{}'=\es;\\
&\fa\al\in\Pi_{\la}\cap\Pi'\quad\quad&\ha{\la|\al}=1.
\end{aligned}}
\end{lemma}

\begin{proof} Предположим, что $r=2$. Тогда система простых корней $\Pi':=\es\subs\Pi$ принадлежит семейству~$\Pc$ и~удовлетворяет~\eqref{nes}. Применяя к~ней лемму~\ref{PP}, получаем, что $r>2$, и~тем самым приходим к~противоречию.

Следовательно, $r>2$. Пусть $\Pi'\in\Pc$\т произвольная система простых корней. Если она удовлетворяет~\eqref{nes}, то соотношения~\eqref{cPi} имеют место в~силу леммы~\ref{PP}. Если же условие~\eqref{nes} не выполнено, то $\Pi_{\la}\cap\Pi'=\es$, что немедленно влечёт~\eqref{cPi}.
\end{proof}

\begin{lemma}\label{undi} Допустим, что $\la\notin\De$, а~условие~\eqref{elem} не выполняется. Тогда $r>2$, $\Pi_{\la}\subs\pd_r\Pi$,
\begin{align}
&\quad&&\quad&&\quad&&\fa\Pi'\in\Pc&|\Pi_{\la}\cap\Pi'|\le1;&\quad&&\quad&\label{dist}\\
&\quad&&\quad&&\quad&&\fa\al\in\Pi_{\la}&\ha{\la|\al}=1.&\quad&&\quad&\label{uni}
\end{align}
\end{lemma}

\begin{proof} Из леммы~\ref{faP} следует, что $r>2$ и~$\Pi_{\la}\subs\pd_r\Pi$, а~также справедливо соотношение~\eqref{dist}. Теперь, применяя утверждение~\ref{cov} и~(повторно) лемму~\ref{faP}, получаем~\eqref{uni}.
\end{proof}

Приведём более удобную и~фундаментальную трактовку леммы~\ref{undi}.

\begin{theorem}\label{submain} Предположим, что $\la\notin\De$, а~условие~\eqref{elem} не выполняется. Тогда
\begin{nums}{-1}
\item $r>2$\~
\item $\ha{\la|\al}=0$ для всякого $\al\in\Pi\sm(\pd_r\Pi)$\~
\item $\ha{\la|\al}\in\{0;1\}$ для всякого $\al\in\pd_r\Pi$\~
\item на схеме Дынкина системы простых корней~$\Pi$ любым различным корням $\al,\be\in\Pi$, таким что $\ha{\la|\al}=\ha{\la|\be}=1$, соответствуют
вершины, путь между которыми содержит не менее $r-2$ рёбер.
\end{nums}
\end{theorem}

На основе таблицы~\ref{tablId} все случаи, когда $r>2$, $\Pi$\т неразложимая система простых корней порядка $r$ (соотв. $r-1$), $\Pi'\in\Pc$ и~$\Inn{\Pi}{}'\ne\es$, описаны с~указанием подмножества $\Inn{\Pi}{}'\subs\Pi$ и~помещены в~таблицу~\ref{tablIr} (соотв. в~таблицу~\ref{tablI1r}). Далее, все случаи, когда $r>2$, $\Pi$\т неразложимая система простых корней порядка $r$ (соотв. $r-1$) и~$\pd_r\Pi\ne\Pi$, приведены с~указанием подмножества $\pd_r\Pi\subs\Pi$ в~таблице~\ref{tabldr} (соотв. в~таблице~\ref{tabld1r}), составленной на основе таблицы~\ref{tablIr} (соотв. таблицы~\ref{tablI1r}).

\begin{prop}\label{adj} Допустим, что $\la\in\De$. Тогда $\ggt=[\ggt,\ggt]$, а~\ter{простая} линейная алгебра $R(\ggt)$ является присоединённой либо совпадает с~одной из простых линейных алгебр $\ph_1(B_r)$, $\ph_2(C_r)$, $\ph_1(F_4)$ и~$\ph_1(G_2)$.
\end{prop}

\begin{proof} Имеем $\la\in\De\subs\ha{\Pi}\subs\tgt^*$, $\ha{\Pi}=\ba{\{\la\}\cup\Pi}=\tgt^*$, $|\Pi|=\dim\tgt^*=r$, $\ggt=[\ggt,\ggt]$. Кроме того, в~системе корней~$\De$ с~системой простых корней~$\Pi$ корень~$\la$ является доминантным, а~значит, совпадает со старшим корнем либо (в~случае наличия корней двух различных длин) с~наибольшим коротким корнем.
\end{proof}

\begin{prop}\label{Bt} Допустим, что $\Pi=B_r$, $\la\notin\De$, $\ha{\la|\al_r}=0$, а~условие~\eqref{elem} не выполняется. Тогда $\ha{\la|\al_1}\seq\ha{\la|\al_{r-2}}=0$.
\end{prop}

\begin{proof} Предположим, что $\ha{\la|\al_i}\ne0$ для некоторого $i\in\{1\sco r-2\}$.

Система простых корней $\Pi':=\{\al_1\sco\al_{r-2}\}\subs\Pi$ принадлежит семейству~$\Pc$ и~удовлетворяет~\eqref{nes}. Легко видеть, что $\ba{\{\la\}\cup\Pi'}=\ha{\al_r}^{\perp}$. В~пространстве~$\tgt^*$ обозначим через~$\Pi_0$ систему простых корней, соответствующую системе положительных корней $\De^+\cap\ba{\{\la\}\cup\Pi'}$. Имеем $\al:=\al_{r-1}+\al_r\in\De^+$, $h_{\al}=2h_{\al_{r-1}}+h_{\al_r}$,
$\ha{\la|\al}=2\ha{\la|\al_{r-1}}+\ha{\la|\al_r}$, а~$\Pi_0$ есть неразложимая система простых корней типа~$B_{r-1}$, включающая в~себя простые корни $\al_1\sco\al_{r-2},\al$ в~стандартном порядке, откуда $\pd\Pi_0\subs\{\al_1,\al\}$. В~силу леммы~\ref{PP}, $i=1$, $\ha{\la|\al_1}=1$, $0=\ha{\la|\al_2}\seq\ha{\la|\al_{r-2}}=\ha{\la|\al}=2\ha{\la|\al_{r-1}}+\ha{\la|\al_r}$, и, как следствие, $\ha{\la|\al_1}=1$, $0=\ha{\la|\al_2}\seq\ha{\la|\al_{r-2}}=\ha{\la|\al_{r-1}}=\ha{\la|\al_r}$, $\la=\ph_1\in\De$. Получили противоречие.
\end{proof}

\begin{prop}\label{Ct} Допустим, что $\Pi=C_r$, $\la\notin\De$, $\ha{\la|\al_r}=0$, а~условие~\eqref{elem} не выполняется. Тогда $\la=\ph_1$ либо $\ha{\la|\al_1}\seq\ha{\la|\al_{r-2}}=0$.
\end{prop}

\begin{proof} Предположим, что $\ha{\la|\al_i}\ne0$ для некоторого $i\in\{1\sco r-2\}$.

Система простых корней $\Pi':=\{\al_1\sco\al_{r-2}\}\subs\Pi$ принадлежит семейству~$\Pc$ и~удовлетворяет~\eqref{nes}. Легко видеть, что $\ba{\{\la\}\cup\Pi'}=\ha{\al_r}^{\perp}$. В~пространстве~$\tgt^*$ обозначим через~$\Pi_0$ систему простых корней, соответствующую системе положительных корней $\De^+\cap\ba{\{\la\}\cup\Pi'}$. Имеем $\al:=2\al_{r-1}+\al_r\in\De^+$, $h_{\al}=h_{\al_{r-1}}+h_{\al_r}$,
$\ha{\la|\al}=\ha{\la|\al_{r-1}}+\ha{\la|\al_r}$, а~$\Pi_0$ есть неразложимая система простых корней типа~$C_{r-1}$, включающая в~себя простые корни $\al_1\sco\al_{r-2},\al$ в~стандартном порядке, откуда $\pd\Pi_0=\{\al_1,\al_2\}$. В~силу леммы~\ref{PP}, $i\in\{1,2\}$, $\ha{\la|\al_j}=\de_{ij}$ для всякого $j=1\sco r-2$, а~также $\ha{\la|\al}=0$. Значит, $\ha{\la|\al_{r-1}}+\ha{\la|\al_r}=\ha{\la|\al}=0$, $\ha{\la|\al_{r-1}}=\ha{\la|\al_r}=0$, что влечёт равенство $\ha{\la|\al_j}=\de_{ij}$ для любого $j=1\sco r$. Тем самым мы получили, что $\la=\ph_i\in\{\ph_1,\ph_2\}$ и, поскольку $\ph_2\in\De$, $\la=\ph_1$.
\end{proof}

\begin{prop}\label{Dt} Допустим, что $r>3$, $\Pi=D_r$, $\la\notin\De\cup\{\ph_1\}$, $\ha{\la|\al_{r-1}}=\ha{\la|\al_r}$, а~условие~\eqref{elem} не выполняется. Тогда $\ha{\la|\al_1}\seq\ha{\la|\al_{r-2}}=0$.
\end{prop}

\begin{proof} Предположим, что $\ha{\la|\al_i}\ne0$ для некоторого $i\in\{1\sco r-2\}$.

Система простых корней $\Pi':=\{\al_1\sco\al_{r-2}\}\subs\Pi$ принадлежит семейству~$\Pc$ и~удовлетворяет~\eqref{nes}, причём $\ba{\{\la\}\cup\Pi'}=\ha{\al_{r-1}-\al_r}^{\perp}$. Далее, в~пространстве~$\tgt^*$ обозначим через~$\Pi_0$ систему простых корней, соответствующую системе положительных корней $\De^+\cap\ba{\{\la\}\cup\Pi'}$. Для корня $\al:=\al_{r-2}+\al_{r-1}+\al_r\in\De^+$ имеем $h_{\al}=h_{\al_{r-2}}+h_{\al_{r-1}}+h_{\al_r}$, $\ha{\la|\al}=\ha{\la|\al_{r-2}}+\ha{\la|\al_{r-1}}+\ha{\la|\al_r}$, а~$\Pi_0$ есть неразложимая система простых корней типа~$D_{r-1}$, включающая в~себя простые корни $\al_1\sco\al_{r-2},\al$ в~стандартном порядке, откуда $\pd\Pi_0\subs\{\al_1,\al_{r-2},\al\}$. Согласно лемме~\ref{PP}, $i\in\{1,r-2\}$, $\ha{\la|\al_j}=\de_{ij}$ для всякого $j=1\sco r-2$, а~также $\ha{\la|\al}=0$. Поэтому $\ha{\la|\al_{r-2}}+\ha{\la|\al_{r-1}}+\ha{\la|\al_r}=\ha{\la|\al}=0$, и, как следствие,
$\ha{\la|\al_{r-2}}=\ha{\la|\al_{r-1}}=\ha{\la|\al_r}=0$, $i\ne r-2$, $i=1$, что влечёт равенство $\ha{\la|\al_j}=\de_{1j}$ для любого $j=1\sco r$. Значит, $\la=\ph_1$. Получили противоречие.
\end{proof}

\begin{imp}\label{BCD} Допустим, что $\la\notin\De$, а~условие~\eqref{elem} не выполняется. Тогда
\begin{nums}{-1}
\item $r>2$\~
\item если $\Pi=B_r$, то $\la\in\{\ph_{r-1},\ph_r,\ph_1+\ph_r,\ph_2+\ph_r,\ph_1+\ph_2+\ph_r\}$\~
\item если $r>3$ и~$\Pi=B_r$, то $\la\in\{\ph_{r-1},\ph_r,\ph_1+\ph_r,\ph_2+\ph_r\}$\~
\item если $\Pi=C_r$, то $\la\in\{\ph_1,\ph_{r-1},\ph_r,\ph_1+\ph_r,\ph_2+\ph_r,\ph_1+\ph_2+\ph_r\}$\~
\item если $r>3$ и~$\Pi=C_r$, то $\la\in\{\ph_1,\ph_{r-1},\ph_r,\ph_1+\ph_r,\ph_2+\ph_r\}$\~
\item если $r>3$ и~$\Pi=D_r$, то $\la\in\{\ph_1,\ph_{r-1},\ph_r,\ph_1+\ph_{r-1},\ph_1+\ph_r,\ph_{r-1}+\ph_r\}$\~
\item если $r>4$ и~$\Pi=D_r$, то $\la\in\{\ph_1,\ph_{r-1},\ph_r,\ph_1+\ph_{r-1},\ph_1+\ph_r\}$.
\end{nums}
\end{imp}

\begin{proof} Вытекает из теоремы~\ref{submain} и~предложений \ref{Bt}---\ref{Dt}.
\end{proof}

Теперь, проводя подробный анализ на основе теоремы~\ref{submain}, предложения~\ref{adj}, следствия~\ref{BCD}, а~также таблиц \ref{tabldr} и~\ref{tabld1r}, получаем утверждение теоремы~\ref{main}.

Перейдём к~доказательству теорем \ref{main1} и~\ref{main2}.

Предположим, что \ter{простая} линейная алгебра Ли $R\br{[\ggt,\ggt]}$ является присоединённой либо классической.

Пусть $\hgt_{\R}\subs\ggt_{\R}$\т с.п.о.п. линейной алгебры Ли $\ggt_{\R}\subs\slg(V_{\R})$, а~$\tgt_{[\ggt,\ggt]}\subs[\ggt,\ggt]$\т картановская подалгебра алгебры $[\ggt,\ggt]$.

Следующее утверждение является очевидным.

\begin{stm}\label{rak} Если $\rk\hgt_{\R}=1$, то условие~\eqref{elem} выполняется, а~если $\rk\hgt_{\R}>1$, то условие~\eqref{elem} не выполняется.
\end{stm}

Возможны следующие случаи.

\begin{cass}{adg1} Линейная алгебра Ли $R\br{[\ggt,\ggt]}$ является присоединённой простой линейной алгеброй Ли, алгебра~$\ggt$ имеет одномерный центр, а~тавтологическое представление $\ggt\cln V$ совпадает с~$R+R'$. Линейная алгебра Ли $R\br{[\ggt,\ggt]}$ не является локально транзитивной, вследствие чего $\tgt_{[\ggt,\ggt]}\subs\ggt$ есть с.п.о.п. представления~$R$ алгебры~$\ggt$. Имеем $\dim\tgt_{[\ggt,\ggt]}=\rk[\ggt,\ggt]=r-1\ge1$.
Значит, система корней простой алгебры Ли $[\ggt,\ggt]$ относительно её картановской подалгебры~$\tgt_{[\ggt,\ggt]}$ содержит некоторое подмножество,
линейно порождающее гиперплоскость в~$\tgt_{[\ggt,\ggt]}^*$. Поэтому представление~$R'|_{\tgt_{[\ggt,\ggt]}}$ алгебры~$\tgt_{[\ggt,\ggt]}$ обладает вектором с~одномерной стационарной подалгеброй\~ то же можно сказать и~о~тавтологическом представлении $\ggt\cln V$. Таким образом, условие~\eqref{elem} выполняется.
\end{cass}

\begin{cass}{ort1} Линейная алгебра Ли $R\br{[\ggt,\ggt]}$ совпадает с~линейной алгеброй $\sog_n(\Cbb)$, где $n\in\{2r-2,2r-1\}$ и~$n>4$, алгебра~$\ggt$ имеет одномерный центр, а~тавтологическое представление $\ggt\cln V$ совпадает с~$R+R'$. Далее, линейная алгебра Ли $R\br{[\ggt,\ggt]}$ не является локально транзитивной. Поэтому с.п.о.п. тавтологического представления $\ggt\cln V$ есть не что иное как с.п.о.п. представления $(R+R')|_{[\ggt,\ggt]}$ алгебры $[\ggt,\ggt]$, которая, в~свою очередь, изоморфна алгебре Ли $\sog_{n-2}(\Cbb)$ ранга $r-2$. Итак, $\rk\hgt_{\R}=r-2$. Значит, $\rk\hgt_{\R}=1$ при $n=5$ и~$\rk\hgt_{\R}>1$ иначе.
\end{cass}

\begin{cass}{un1} Тавтологическое представление $\ggt_{\R}\cln V_{\R}$ получается из представления
$\Mat_{(r-1)\times(r-1)}(\Hbb)\oplus\Mat_{1\times1}(\Hbb)\cln\Mat_{(r-1)\times1}(\Hbb),\,(A+B)\cln X\to AX-XB$ ($r>2$) ограничением на подалгебру $\br{\un_{r-1}(\Hbb)\oplus\ib\R E}\subs\br{\Mat_{(r-1)\times(r-1)}(\Hbb)\oplus\Mat_{1\times1}(\Hbb)}$\т прямую сумму подалгебр
$\un_{r-1}(\Hbb)\subs\Mat_{(r-1)\times(r-1)}(\Hbb)$ и~$\ib\R E\subs\Mat_{1\times1}(\Hbb)$. Легко видеть, что $\hgt_{\R}\cong\un_{r-2}(\Hbb)\oplus\R$. Отсюда $\rk\hgt_{\R}=r-1>1$.
\end{cass}

\begin{cass}{adg} Линейная алгебра $\ggt_{\R}\subs\slg(V_{\R})$ является присоединённой простой компактной линейной алгеброй Ли. Подалгебра $\hgt_{\R}\subs\ggt_{\R}$ является максимальной коммутативной подалгеброй алгебры~$\ggt_{\R}$, откуда $\rk\hgt_{\R}=\rk\ggt_{\R}=r>1$.
\end{cass}

\begin{cass}{ort} Линейная алгебра $\ggt_{\R}\subs\slg(V_{\R})$ совпадает с~одной из линейных алгебр $\sog_{2r}(\R)$ ($r>2$), $\sog_{2r+1}(\R)$, $\sug_{r+1}(\Cbb)$, $\un_r(\Cbb)$, $\un_r(\Hbb)$. Подалгебра $\hgt_{\R}\subs\ggt_{\R}$ изоморфна алгебре $\sog_{2r-1}(\R)$, $\sog_{2r}(\R)$, $\sug_r(\Cbb)$, $\un_{r-1}(\Cbb)$, $\un_{r-1}(\Hbb)$ соответственно и~имеет ранг, равный $r$ при $\ggt_{\R}=\sog_{2r+1}(\R)$ и~$r-1$ иначе. Отсюда $\rk\hgt_{\R}=1$ при $\ggt_{\R}=\sug_3(\Cbb)$, $\ggt_{\R}=\un_2(\Cbb)$ либо $\ggt_{\R}=\un_2(\Hbb)$ и~$\rk\hgt_{\R}>1$ в~противном случае.
\end{cass}

Теперь для доказательства теорем \ref{main1} и~\ref{main2} достаточно воспользоваться утверждением~\ref{rak}.

\newpage

\section{Справочный раздел}\label{tables}

\lont{>{$}l<{$}|>{$}l<{$}|>{$}l<{$}|}{}{tablId}{\nc{\Pi}\nc{\pd\Pi}\nc{\Inn{\Pi}}}{%
\an{dAr}
A_r,\quad r\ge6 & 1,2,r-1,r & 3\sco r-2
\n{dB4}
B_r,\quad r=3,4 & 1,r & 2\sco r-1
\n{dBr}
B_r,\quad r\ge5 & 1 & 2\sco r
\n{dCr}
C_r,\quad r\ge3 & 1,2 & 3\sco r
\n{dD6}
D_r,\quad r=4,5,6 & 1,r-1,r & 2\sco r-2
\n{dDr}
D_r,\quad r\ge7 & 1 & 2\sco r
\n{dE6}
E_6 & 1,5 & 2,3,4,6
\n{dE7}
E_7 & 1 & 2\sco7
\n{dE8}
E_8 & \es & 1\sco8
\n{dF4}
F_4 & 1 & 2,3,4
\n{dG2}
G_2 & 1 & 2}

\lont{>{$}l<{$}|>{$}l<{$}|>{$}l<{$}|}{}{tablOm}{\nc{\Pi(\al)}\nc{i}\nc{\br{\Pi\sm\{\al_i\}}(\al)}}{%
\an{Ar}
A_r(1) & r & A_{r-1}(1)
\n{Ar2r}
A_r(2) & r & A_{r-1}(2)
\n{Ar21}
A_r(2) & 1 & A_{r-1}(1)
\n{A531}
A_5(3) & 1 & A_4(2)
\n{Br}
B_r(1) & r & A_{r-1}(1)
\n{Brr}
B_r(r),\quad r=3,4 & 1 & B_{r-1}(r-1)
\n{Cr}
C_r(1) & r & A_{r-1}(1)
\n{Cr2r}
C_r(2) & r & A_{r-1}(2)
\n{Cr21}
C_r(2) & 1 & C_{r-1}(1)
\n{Dr}
D_r(1) & r & A_{r-1}(1)
\n{Drrr}
D_r(r),\quad r=5,6 & r-1 & A_{r-1}(r-1)
\n{Drr1}
D_r(r),\quad r=5,6 & 1 & D_{r-1}(r-1)
\n{Er}
E_r(1),\quad r=6,7 & r & A_{r-1}(1)
\n{Er1}
E_r(1),\quad r=6,7 & r-1 & D_{r-1}(1)
\n{F4}
F_4(1) & 4 & C_3(1)}

\lont{>{$}l<{$}|>{$}l<{$}|>{$}l<{$}|>{$}l<{$}|}{}{tablIr}{\nc{\Pi}\mc{2}{\Pi'}\nc{\Inn{\Pi}{}'}}{%
\an{I0Ar3}
A_r,\quad r\ge8 & 3\sco r & A_{r-2} & 5\sco r-2
\n{I0Ar2}
A_r,\quad r\ge8 & 2\sco r-1 & A_{r-2} & 4\sco r-3
\n{I0Ar1}
A_r,\quad r\ge8 & 1\sco r-2 & A_{r-2} & 3\sco r-4
\n{I0B63}
B_r,\quad r=5,6 & 3\sco r & B_{r-2} & 4\sco r-1
\n{I0Br3}
B_r,\quad r\ge7 & 3\sco r & B_{r-2} & 4\sco r
\n{I0Br2}
B_r,\quad r\ge8 & 2\sco r-1 & A_{r-2} & 4\sco r-3
\n{I0Br1}
B_r,\quad r\ge8 & 1\sco r-2 & A_{r-2} & 3\sco r-4
\n{I0Cr3}
C_r,\quad r\ge5 & 3\sco r & C_{r-2} & 5\sco r
\n{I0Cr2}
C_r,\quad r\ge8 & 2\sco r-1 & A_{r-2} & 4\sco r-3
\n{I0Cr1}
C_r,\quad r\ge8 & 1\sco r-2 & A_{r-2} & 3\sco r-4
\n{I0D83}
D_r,\quad r=6,7,8 & 3\sco r & D_{r-2} & 4\sco r-2
\n{I0Dr3}
D_r,\quad r\ge9 & 3\sco r & D_{r-2} & 4\sco r
\n{I0Dr2}
D_r,\quad r\ge8 & 2\sco r-1 & A_{r-2} & 4\sco r-3
\n{I0Dr2r}
D_r,\quad r\ge8 & 2\sco r-2,r & A_{r-2} & 4\sco r-3
\n{I0Dr1}
D_r,\quad r\ge8 & 1\sco r-2 & A_{r-2} & 3\sco r-4
\n{I0E6}
E_6 & 2,3,4,6 & D_4 & 3
\n{I0E73}
E_7 & 3\sco7 & D_5 & 4,5
\n{I0E72}
E_7 & 2\sco5,7 & D_5 & 3,4
\n{I0E83}
E_8 & 3\sco8 & E_6 & 4,5,6,8
\n{I0E828}
E_8 & 2\sco6,8 & D_6 & 3,4,5
\n{I0E818}
E_8 & 1\sco5,8 & A_6 & 3,4
\n{I0E82}
E_8 & 2\sco7 & A_6 & 4,5
\n{I0E81}
E_8 & 1\sco6 & A_6 & 3,4}

\lont{>{$}l<{$}|>{$}l<{$}|}{}{tabldr}{\nc{\Pi}\nc{\pd_r\Pi}}{%
\an{d0Ar}
A_r,\quad r\ge8 & 1,2,r-1,r
\n{d0B6}
B_r,\quad r=5,6 & 1,2,3,r
\n{d0B7}
B_7 & 1,2,3
\n{d0Br}
B_r,\quad r\ge8 & 1,2
\n{d0C7}
C_r,\quad r=5,6,7 & 1,2,3,4
\n{d0Cr}
C_r,\quad r\ge8 & 1,2
\n{d0D7}
D_r,\quad r=6,7 & 1,2,3,r-1,r
\n{d0D8}
D_8 & 1,2,7,8
\n{d0Dr}
D_r,\quad r\ge9 & 1,2
\n{d0E6}
E_6 & 1,2,4,5,6
\n{d0E7}
E_7 & 1,2,6,7
\n{d0E8}
E_8 & 1,2,7}

\lont{>{$}l<{$}|>{$}l<{$}|>{$}l<{$}|>{$}l<{$}|}{}{tablI1r}{\nc{\Pi}\mc{2}{\Pi'}\nc{\Inn{\Pi}{}'}}{%
\an{I1Ar2}
A_{r-1},\quad r\ge8 & 2\sco r-1 & A_{r-2} & 4\sco r-3
\n{I1Ar1}
A_{r-1},\quad r\ge8 & 1\sco r-2 & A_{r-2} & 3\sco r-4
\n{I1B52}
B_{r-1},\quad r=5,6 & 2\sco r-1 & B_{r-2} & 3\sco r-2
\n{I1Br2}
B_{r-1},\quad r\ge7 & 2\sco r-1 & B_{r-2} & 3\sco r-1
\n{I1Br1}
B_{r-1},\quad r\ge8 & 1\sco r-2 & A_{r-2} & 3\sco r-4
\n{I1Cr2}
C_{r-1},\quad r\ge5 & 2\sco r-1 & C_{r-2} & 4\sco r-1
\n{I1Cr1}
C_{r-1},\quad r\ge8 & 1\sco r-2 & A_{r-2} & 3\sco r-4
\n{I1D72}
D_{r-1},\quad r=6,7,8 & 2\sco r-1 & D_{r-2} & 3\sco r-3
\n{I1Dr2}
D_{r-1},\quad r\ge9 & 2\sco r-1 & D_{r-2} & 3\sco r-1
\n{I1Dr1}
D_{r-1},\quad r\ge8 & 1\sco r-2 & A_{r-2} & 3\sco r-4
\n{I1Dr1r}
D_{r-1},\quad r\ge8 & 1\sco r-3,r-1 & A_{r-2} & 3\sco r-4
\n{I1E62}
E_6 & 2\sco6 & D_5 & 3,4
\n{I1E61}
E_6 & 1,2,3,4,6 & D_5 & 2,3
\n{I1E72}
E_7 & 2\sco7 & E_6 & 3,4,5,7
\n{I1E717}
E_7 & 1\sco5,7 & D_6 & 2,3,4
\n{I1E71}
E_7 & 1\sco6 & A_6 & 3,4
\n{I1E82}
E_8 & 2\sco8 & E_7 & 3\sco8
\n{I1E818}
E_8 & 1\sco6,8 & D_7 & 2\sco6,8
\n{I1E81}
E_8 & 1\sco7 & A_7 & 3,4,5
\n{I1F42}
F_4 & 2,3,4 & B_3 & 3
\n{I1F41}
F_4 & 1,2,3 & C_3 & 3}

\lont{>{$}l<{$}|>{$}l<{$}|}{}{tabld1r}{\nc{\Pi}\nc{\pd_r\Pi}}{%
\an{d1Ar}
A_{r-1},\quad r\ge8 & 1,2,r-2,r-1
\n{d1B6}
B_{r-1},\quad r=5,6 & 1,2,r-1
\n{d1Br}
B_{r-1},\quad r\ge7 & 1,2
\n{d1C7}
C_{r-1},\quad r=5,6,7 & 1,2,3
\n{d1Cr}
C_{r-1},\quad r\ge8 & 1,2
\n{d1D7}
D_{r-1},\quad r=6,7,8 & 1,2,r-2,r-1
\n{d1Dr}
D_{r-1},\quad r\ge9 & 1,2
\n{d1E6}
E_6 & 1,5,6
\n{d1E7}
E_7 & 1,6
\n{d1E8}
E_8 & 1
\n{d1F4}
F_4 & 1,2,4}

\end{document}